\documentclass[a4paper,UKenglish,cleveref, autoref, thm-restate]{lipics-v2021}

\newcommand{\Z}{\mathbb{Z}}
\newcommand{\solid}{{\rule[0.36pt]{3.5pt}{3.5pt}}}
\newcommand{\ev}{\operatorname{ev}}
\newcommand{\suchthat}{\ |\ }
\newtheorem{notation}[theorem]{Notation}
\newtheorem{construction}[theorem]{Construction}

\usepackage{quiver}
\usepackage{newunicodechar}

\usepackage{color}
\definecolor{keywordcolor}{rgb}{0.7, 0.1, 0.1}   
\definecolor{tacticcolor}{rgb}{0.0, 0.1, 0.6}    
\definecolor{commentcolor}{rgb}{0.4, 0.4, 0.4}   
\definecolor{symbolcolor}{rgb}{0.0, 0.1, 0.6}    
\definecolor{sortcolor}{rgb}{0.1, 0.5, 0.1}      
\definecolor{attributecolor}{rgb}{0.7, 0.1, 0.1} 

\usepackage{xspace}
\newcommand{\lean}[1]{\lstinline{#1}\xspace} 

\title{Towards solid abelian groups: \\ A formal proof of Nöbeling's theorem}

\titlerunning{Towards solid abelian groups: A formal proof of Nöbeling's theorem}

\author{Dagur Asgeirsson}{University of Copenhagen, Denmark \and \url{http://dagur.sites.ku.dk}}{dagur@math.ku.dk}{https://orcid.org/0000-0003-3002-0320}{The author was supported by the Danish National Research Foundation (DNRF) through the ``Copenhagen Center for Geometry and Topology'' under grant no. DNRF151.}

\authorrunning{D. Asgeirsson}

\Copyright{Dagur Asgeirsson} 



\category{} 

\relatedversion{} 

\supplement{The code described has all been integrated into Mathlib.}
\supplementdetails[linktext={Mathlib/Topology/Category/Profinite/Nobeling.lean}]{Software}{https://github.com/leanprover-community/mathlib4/blob/ba9f2e5baab51310883778e1ea3b48772581521c/Mathlib/Topology/Category/Profinite/Nobeling.lean} 

\acknowledgements{First and foremost, I would like to thank Johan Commelin for encouraging me to start seriously working on this project when we were both in Banff attending the workshop on formalisation of cohomology theories last year. I had useful discussions related to this work with Johan, Kevin Buzzard, Adam Topaz, and Dustin Clausen. I am indebted to all four of them for providing helpful feedback on earlier drafts of this paper. Any project formalising serious mathematics in Lean depends on Mathlib, and this one is no exception. I am grateful to the Lean community as a whole for building and maintaining such a useful mathematical library, as well as providing an excellent forum for Lean-related discussions through the Zulip chat. Finally, I would like to thank the anonymous referees for helpful feedback.}

\nolinenumbers 

\begin{document}

\maketitle

\begin{abstract}
    Condensed mathematics, developed by Clausen and Scholze over the last few years, is a new way of studying the interplay between algebra and geometry. It replaces the concept of a topological space by a more sophisticated but better-behaved idea, namely that of a condensed set. Central to the theory are solid abelian groups and liquid vector spaces, analogues of complete topological groups.

    Nöbeling's theorem, a surprising result from the 1960s about the structure of the abelian group of continuous maps from a profinite space to the integers, is a crucial ingredient in the theory of solid abelian groups; without it one cannot give any nonzero examples of solid abelian groups. We discuss a recently completed formalisation of this result in the Lean theorem prover, and give a more detailed proof than those previously available in the literature. The proof is somewhat unusual in that it requires induction over ordinals -- a technique which has not previously been used to a great extent in formalised mathematics.
\end{abstract}

\section{Introduction}

Nöbeling's theorem says that the abelian group $C(S,\Z)$ of continuous maps from a profinite space $S$ to the integers, is a free abelian group. In fact, the original statement \cite[Satz 1]{nobeling} is the more general result that bounded maps from any set to the integers form a free abelian group, but this special case has recently been applied \cite[Theorem 5.4]{condensed} in the new field of condensed mathematics (see also \cite{analytic, complex, pyknotic}). 

We report on a recently completed formalisation of this theorem using the Lean 4 theorem prover \cite{lean4}, building on its \emph{Mathlib} library of formalised mathematics (which was recently ported from Lean 3, see \cite{lean3, mathlib}). The proof uses the well-ordering principle and a tricky induction over ordinals. This is the first use of the induction principle for ordinals in Mathlib outside the directory containing the theory of ordinals. Often, one can replace such transfinite constructions by appeals to Zorn's lemma. The author is not aware of any proof of Nöbeling's theorem that does this, or otherwise avoids induction over ordinals\footnote{The proof does not use any ordinal arithmetic. However, it crucially uses the principle of induction over ordinals with a case split between limit ordinals and successor ordinals.}. 

When formalising a nontrivial proof, one inevitably makes an effort to organise the argument carefully. One purpose of this paper is to give a well-organised and detailed proof of Nöbeling's theorem, written in conventional mathematical language, which is essentially a by-product of the formalisation effort. This will hopefully be a more accessible proof than those that already exist in the literature; the one in \cite{nobeling} is in German, while the proofs of the result in \cite{fuchs, condensed} are the same argument as the one presented here, but in significantly less detail. This is the content of section \ref{sec:theorem}; some mathematical prerequisites are found in section \ref{sec:preliminaries}. 

In section \ref{sec:motivation} we give more details about the connection to condensed mathematics and in sections \ref{sec:formalisation}  and \ref{sec:condensed_mathlib} we discuss the formalisation process and the integration into Mathlib. 

Throughout the text, we use the symbol ``\faExternalLink*'' for external links, usually directly to the source code for the corresponding theorems and definitions in Mathlib. In order for the links to stay usable, they are all to a fixed commit to the master branch (the most recent one at the time of writing). 

Mathlib is a growing library of mathematics formalised in Lean. All material is maintained continuously by a team of experts. There is a big emphasis on unity, meaning that there is \emph{one} official definition of every concept, and it is the job of contributors to provide proofs that alternative definitions are equivalent. All the code in this project has been integrated into Mathlib; a process that took quite some time, as high standards are demanded of code that enters the library. However, it is an important part of formalisation to get the code into Mathlib, because doing so means that it stays usable to others in the future. 

\section{Motivation}\label{sec:motivation}

Condensed mathematics \cite{condensed, analytic, complex} is a new theory developed by Clausen and Scholze (and independently by Barwick and Haine, who called the theory \emph{pyknotic sets} \cite{pyknotic}). It has the purpose of generalising topology in a way that gives better categorical properties, which is desirable e.g. when the objects have both a topological and an algebraic structure. Condensed objects\footnote{This notion was first formalised in the \emph{Liquid Tensor Experiment} \cite{LTE, LTE_challenge}, see section \ref{sec:condensed_mathlib} for a more detailed discussion.} can be described as sheaves on a certain site of profinite spaces. A topological abelian group $A$ can be regarded as a condensed abelian group with $S$-valued points $C(S,A)$ for profinite spaces $S$.  Discrete abelian groups such as $\Z$ are important examples of topological abelian groups. There is a useful characterisation of discrete condensed sets (which leads to the same characterisation for more general condensed objects such as condensed abelian groups), which has been formalised in Lean 3 by the author in \cite{discrete}. 

The discreteness characterisation can be stated somewhat informally as follows: A condensed set $X$ is discrete if and only if for every profinite space $S = \varprojlim_i S_i$ (written as a cofiltered limit of finite discrete spaces), the natural map 
\[
    \varinjlim_i X(S_i) \to X(S)    
\]
is an isomorphism. 

There is a notion of completeness of condensed abelian groups, called being \emph{solid} \cite[Definition 5.1]{condensed}. For the convenience of the reader, we give the informal definition here in Definition \ref{def:solid}. First, we need to recall two facts about condensed abelian groups:
\begin{itemize}
    \item The category of condensed abelian groups has all limits.
    \item The forgetful functor from condensed abelian groups to condensed sets has a left adjoint, denoted by $\Z[-]$ (adopted from the analogous relationship between the category of sets and the category of abelian groups). 
\end{itemize}

\begin{definition}\label{def:solid}
    Let $S = \varprojlim_i S_i$ be a profinite space and define a condensed abelian group as follows:
    \[
        \Z[S]^\solid := \varprojlim_i \Z[S_i]
    \]
    There is a natural map $\Z[S] \to \Z[S]^\solid$, and we say that a condensed abelian group $A$ is \emph{solid} if for every profinite space $S$ and every morphism $f: \Z[S] \to A$ of condensed abelian groups, there is a unique morphism $g : \Z[S]^\solid \to A$ making the obvious triangle commute\footnote{This definition has also been formalised by the author in Lean 3 in \cite{solid}}.
\end{definition}

Using the discreteness characterisation and Nöbeling's theorem, one can prove that for every profinite space $S$, there is a set $I$ and an isomorphism of condensed abelian groups 
\[
    \Z[S]^\solid \cong  \prod_{i \in I} \Z.
\]
This structural result is essential to developing the theory of solid abelian groups. Without it one cannot even prove the existence of a nontrivial solid abelian group.

Since the proof of Nöbeling's theorem has nothing to do with condensed mathematics, people studying the theory might be tempted to skip the proof and use Nöbeling's theorem as a black box. Now that it has been formalised, they can do this with a better conscience. On the other hand, people interested in understanding the proof might want to turn to sections \ref{sec:preliminaries} and \ref{sec:theorem} of this paper for a more detailed account.

\section{Preliminaries}\label{sec:preliminaries}
For ease of reference, we collect in this section some prerequisites for the proof of Nöbeling's theorem. Most of them were already in Mathlib. 

\subsection{Order theory}

\begin{definition}\href{https://github.com/leanprover-community/mathlib4/blob/ba9f2e5baab51310883778e1ea3b48772581521c/Mathlib/Order/Directed.lean#L43-L44}{\faExternalLink*}
    Let $I$ and $X$ be sets and let $r$ be a binary relation on $X$. An $I$-indexed family $(x_i)$ in $X$ is \emph{directed} if for all $i, j \in I$, there exists $k \in I$ such that $r(x_i, x_k)$ and $r(x_j,x_k)$.  
\end{definition}

\begin{lemma}\label{directed_of_sup} \href{https://github.com/leanprover-community/mathlib4/blob/ba9f2e5baab51310883778e1ea3b48772581521c/Mathlib/Order/Directed.lean#L174-L177}{\faExternalLink*}
    A monotone map on a poset with a join operation (i.e. a least upper bound of two elements) is directed. 
\end{lemma}
\begin{remark}
    Taking the union of two sets is an example of a join operation.
\end{remark}

\begin{definition}\href{https://github.com/leanprover-community/mathlib4/blob/ba9f2e5baab51310883778e1ea3b48772581521c/Mathlib/CategoryTheory/Filtered/Basic.lean#L63-L86}{\faExternalLink*}
    A category $\mathcal{C}$ is \emph{filtered} if it satisfies the following three conditions 
    \begin{enumerate}[(i)]
        \item $\mathcal{C}$ is nonempty.
        \item For all objects $X,Y$, there exists an object $Z$ and morphisms $f : X \to Z$ and $g : Y \to Z$.
        \item For all objects $X,Y$ and all morphisms $f,g : X \to Y$, there exists an object $Z$ and a morphism $h : Y \to Z$ such that $h \circ f = h \circ g$. 
    \end{enumerate}
    A category is \emph{cofiltered} if the opposite category is filtered.
\end{definition}

\begin{remark}
    A poset is filtered if and only if it is nonempty and directed.
\end{remark}

\begin{remark}
    The poset of finite subsets of a given set is filtered.
\end{remark}

\subsection{Linear Independence}

\begin{lemma}\label{linearIndependent_iUnion_of_directed}\href{https://github.com/leanprover-community/mathlib4/blob/ba9f2e5baab51310883778e1ea3b48772581521c/Mathlib/LinearAlgebra/LinearIndependent.lean#L517-L528}{\faExternalLink*}
    If $(X_i)$ is a family of linearly independent subsets of a module over a ring $R$, which is directed with respect to the subset relation, then its union is linearly independent.
\end{lemma}

\begin{lemma}\label{linearIndependent_leftExact}\href{https://github.com/leanprover-community/mathlib4/blob/ba9f2e5baab51310883778e1ea3b48772581521c/Mathlib/Algebra/Category/ModuleCat/Free.lean#L53-L70}{\faExternalLink*}
    Suppose we have a commutative diagram 
    \[
    \begin{tikzcd}
        0 & N & M & P \\
        & I & {I \sqcup J} & J
        \arrow[from=1-1, to=1-2]
        \arrow["f", from=1-2, to=1-3]
        \arrow["g", from=1-3, to=1-4]
        \arrow[hook, from=2-2, to=2-3]
        \arrow[hook', from=2-4, to=2-3]
        \arrow["v"{description}, from=2-2, to=1-2]
        \arrow["u"{description}, from=2-3, to=1-3]
        \arrow["w"{description}, from=2-4, to=1-4]
    \end{tikzcd}\]
    where $N,M,P$ are modules over a ring $R$, the top row is exact, and the bottom maps are the inclusion maps. If $v$ and $w$ are linearly independent, then $u$ is linearly independent. 
\end{lemma}

\subsection{Cantor's intersection theorem}

\begin{theorem}\label{cantor_inter}\href{https://github.com/leanprover-community/mathlib4/blob/ba9f2e5baab51310883778e1ea3b48772581521c/Mathlib/Topology/Compactness/Compact.lean#L289-L303}{\faExternalLink*}
    \emph{Cantor's intersection theorem}. If $(Z_i)_{i \in I}$ is a nonempty family of nonempty, closed and compact subsets of a topological space $X$, which is directed with respect to the \emph{superset} relation $(V, W) \mapsto V \supseteq W$, then the intersection $\bigcap_{i \in I} Z_i$ is nonempty.
\end{theorem}

\begin{remark}
    Cantor's intersection theorem is often stated only for the special case of decreasing nested sequences of nonempty compact, closed subsets. The generalisation above can be proved by slightly modifying the standard proof of that special case. 
\end{remark}

\subsection{Cofiltered limits of profinite spaces}

\begin{definition}\href{https://github.com/leanprover-community/mathlib4/blob/ba9f2e5baab51310883778e1ea3b48772581521c/Mathlib/Topology/Category/Profinite/Basic.lean#L46-L51}{\faExternalLink*}
    A \emph{profinite space} is a totally disconnected compact Hausdorff space.
\end{definition}

\begin{lemma}\href{https://github.com/leanprover-community/mathlib4/blob/ba9f2e5baab51310883778e1ea3b48772581521c/Mathlib/Topology/Separation.lean#L2338-L2344}{\faExternalLink*}
    Every profinite space has a basis of clopen subsets.
\end{lemma}

\begin{lemma}\href{https://github.com/leanprover-community/mathlib4/blob/ba9f2e5baab51310883778e1ea3b48772581521c/Mathlib/Topology/Separation.lean#L2293-L2305}{\faExternalLink*}
    Every profinite space is totally separated, i.e. any two distinct points can be separated by clopen neighbourhoods. 
\end{lemma}

\begin{remark}
    A topological space is profinite if and only if it can be written as a cofiltered limit of finite discrete spaces. See section \ref{sec:condensed_mathlib} for a further discussion.
\end{remark}

\begin{lemma}\label{jointly_surjective}\href{https://github.com/leanprover-community/mathlib4/blob/ba9f2e5baab51310883778e1ea3b48772581521c/Mathlib/Topology/Category/Profinite/CofilteredLimit.lean#L211-L250}{\faExternalLink*}
    Any continuous map from a cofiltered limit of profinite spaces to a discrete space factors through one of the components. 
\end{lemma}

\begin{remark}
    In particular, a continuous map from a profinite space 
    \[
        S = \varprojlim_i S_i
    \] 
    to a discrete space factors through one of the finite quotients $S_i$. 
\end{remark}

\section{The theorem}\label{sec:theorem}
This section is devoted to proving

\begin{theorem}\label{Nobeling}\href{https://github.com/leanprover-community/mathlib4/blob/ba9f2e5baab51310883778e1ea3b48772581521c/Mathlib/Topology/Category/Profinite/Nobeling.lean#L1827-L1832}{\faExternalLink*}
    (\textit{Nöbeling's theorem}). Let $S$ be a profinite space. Then the abelian group $C(S, \Z)$ of continuous maps from $S$ to $\Z$ is free.
\end{theorem}

We can immediately reduce this to proving Lemma \ref{NobelingClosed} below as follows:
Let $I$ denote the set of clopen subsets of $S$. Then the map 
\[
    S \to  \prod_{i \in I} \{0,1\}
\]
whose $i$-th projection is given by the indicator function of the clopen subset $i$ is a closed embedding.

\begin{lemma}\label{NobelingClosed}\href{https://github.com/leanprover-community/mathlib4/blob/ba9f2e5baab51310883778e1ea3b48772581521c/Mathlib/Topology/Category/Profinite/Nobeling.lean#L1769-L1773}{\faExternalLink*}
    Let $I$ be a set and let $S$ be a closed subset of $ \prod_{i \in I} \{0,1\}$. Then $C(S,\Z)$ is a free abelian group.
\end{lemma}

To prove Lemma \ref{NobelingClosed}, we need to construct a basis of $C(S, \Z)$. Our proposed basis is defined as follows: 
\begin{itemize}
    \item Choose a well-ordering on $I$. 
    \item 
        Let $e_{S, i} \in C(S,\Z)$ denote the composition
        \[\begin{tikzcd}
            S & { \prod_{i \in I}\{0,1\}} & {\{0,1\}} & {\Z }
            \arrow[hook, from=1-1, to=1-2]
            \arrow["{p_i}", two heads, from=1-2, to=1-3]
            \arrow[hook, from=1-3, to=1-4]
        \end{tikzcd}\]
        where $p_i$ denotes the $i$-th projection map, and the other two maps are the obvious inclusions. 
    \item 
        Let $P$ denote the set of finite, strictly decreasing sequences in $I$. Order these lexicographically.
    \item 
        Let $\ev_S : P \to C(S,\Z)$ denote the map 
        \[
            (i_1,\cdots,i_r) \mapsto e_{S, i_1} \cdots e_{S, i_r}.
        \] 
    \item 
        For $p \in P$, let $\Sigma_S(p)$ denote the span in $C(S,\Z)$ of the set 
        \[
            \ev_S \left( \{q \in P \suchthat q < p\} \right).
        \]
    \item 
        Let $E(S)$ denote the subset of $P$ consisting of those elements whose evaluation cannot be written as a linear combination of evaluations of smaller elements of $P$, i.e. 
        \[
            E(S) := \{p \in P \suchthat \ev_S (p) \notin \Sigma_S(p)\}.
        \]
\end{itemize}
In Subsection \ref{subsec_span} we prove that the set $\ev_S \left(E(S)\right)$ spans $C(S,\Z)$, and in Subsection \ref{subsec_indep} we prove that the family 
\[
    \ev_S: E(S) \to C(S,\Z)
\]
is linearly independent, concluding the proof of Nöbeling's theorem. Subsection \ref{sec:notation_and_generalities} defines some notation which will be convenient for bookkeeping in the subsequent proof.

\subsection{Notation and generalities}\label{sec:notation_and_generalities}

For a subset $J$ of $I$ we denote by 
\[
    \pi_J :  \prod_{i \in I} \{0,1\} \to  \prod_{i \in I} \{0,1\} 
\]
the map whose $i$-th projection is $p_i$ if $i \in J$, and $0$ otherwise. These maps are continuous, and since source and target are compact Hausdorff spaces, they are also closed. 
Given a subset $S \subseteq  \prod_{i \in I} \{0,1\}$, we let 
\[
    S_J := \pi_J (S).
\] 
We can regard $I$ with its well-ordering as an ordinal. Then $I$ is the set of all strictly smaller ordinals. Given an ordinal $\mu$, we let 
\[
    \pi_\mu := \pi_{\{i \in I \suchthat i < \mu\}}
\] 
and 
\[
    S_\mu := S_{\{i \in I \suchthat i < \mu\}}.
\]

These maps induce injective $\Z$-linear maps 
\[
    \pi_J^* : C(S_J, \Z) \to C(S,\Z)
\] 
by precomposition.

Recall that we have defined $P$ as the set of finite, strictly decreasing sequences in $I$, ordered lexicographically. We will use this notation throughout the proof of Nöbeling's theorem. 

\begin{lemma}\label{Products.eval_eq}\href{https://github.com/leanprover-community/mathlib4/blob/ba9f2e5baab51310883778e1ea3b48772581521c/Mathlib/Topology/Category/Profinite/Nobeling.lean#L378-L392}{\faExternalLink*}
    For $p \in P$ and $x \in S$, we have 
    \[
        \ev_S(p)(x) = \begin{cases} 1 & \text{ if } \forall i \in p,\ x_i = 1 \\ 
            0 & \text{ otherwise.}  \end{cases} 
    \]
\end{lemma}
\begin{proof}
    Obvious.
\end{proof}

\begin{lemma}\label{Products.evalFacProp}\href{https://github.com/leanprover-community/mathlib4/blob/ba9f2e5baab51310883778e1ea3b48772581521c/Mathlib/Topology/Category/Profinite/Nobeling.lean#L394-L403}{\faExternalLink*}
    Let $J$ be a subset of $I$ and let $p \in P$ be such that $i \in p$ implies $i \in J$. Then $\pi_J^*(\ev_{S_J} (p)) = \ev_S (p)$. 
\end{lemma}
\begin{proof}
    Since $i \in p$ implies $i \in J$, we have 
    \[
        x_i = \pi_J^*(x)_i
    \] 
    for all $x \in S$ and $i \in p$. The result now follows from Lemma \ref{Products.eval_eq}. 
\end{proof}

\begin{remark}\label{rem:prop_of_isGood}\href{https://github.com/leanprover-community/mathlib4/blob/ba9f2e5baab51310883778e1ea3b48772581521c/Mathlib/Topology/Category/Profinite/Nobeling.lean#L414-L428}{\faExternalLink*}
    The hypothesis in Lemma \ref{Products.evalFacProp} holds in particular if $p \in E(S_J)$. Indeed, suppose $i \in p$, then if $i \notin J$, we have $\ev_{S_J} (p) = 0$. 
\end{remark}

\begin{lemma}\label{good_mono}\href{https://github.com/leanprover-community/mathlib4/blob/ba9f2e5baab51310883778e1ea3b48772581521c/Mathlib/Topology/Category/Profinite/Nobeling.lean#L964-L973}{\faExternalLink*}
    If $\mu', \mu$ are ordinals satisfying $\mu' < \mu$, then $E(S_{\mu'}) \subseteq E(S_{\mu})$.
\end{lemma}
\begin{proof}
    Let $p \in E(S_{\mu'})$. Then every entry of $p$ is $<\mu'$, and it suffices to show that if 
    \[
        \ev_{S_\mu}(p) = \pi_{\mu'}^*(\ev_{S_{\mu'}}(p))
    \] 
    is in the span of 
    \[
        \ev_{S_{\mu}} \left( \{q \in P \suchthat q < p\} \right) 
        = \pi_{\mu'}^* \left( \ev_{S_{\mu'}} \left( \{q \in P \suchthat q < p\} \right) \right)
    \]
    then $\ev_{S_{\mu'}}(p)$ is in the span of $\ev_{S_{\mu'}} \left( \{q \in P \suchthat q < p\} \right)$. This follows by injectivity of $\pi_{\mu'}^*$. 
\end{proof}

\subsection{Span}\label{subsec_span}

The following series of lemmas proves that $\ev_S \left(E(S)\right)$ spans $C(S,\Z)$.

\begin{lemma}\label{P_well_ordered}
    The set $P$ is well-ordered.
\end{lemma}
\begin{proof}[Proof sketch]
    Suppose not. Take a strictly decreasing sequence $(p_n)$ in $P$. Let $a_n$ denote first term of $p_n$. Then $(a_n)$ is a decreasing sequence in $I$ and hence eventually constant. Denote its limit by $a$. Let $q_n = p_n \setminus {a_n}$. Then there exists an $N$ such that $(q_n)_{n \geq N}$ is a strictly decreasing sequence in $P$ and we can repeat the process of taking the indices of the first factors, get a decreasing sequence in $I$ whose limit is strictly smaller than $a$. Continuing this way, we get a strictly decreasing sequence in $I$, a contradiction.
\end{proof}

\begin{remark}
    The proof sketch of Lemma \ref{P_well_ordered} above is ill-suited for formalisation. Kim Morrison gave a formalised proof \href{https://github.com/leanprover-community/mathlib4/pull/6361}{\faExternalLink*}, following similar ideas to those above, which used close to 300 lines of code. A few days later, Junyan Xu found a proof \href{https://github.com/leanprover-community/mathlib4/pull/6432}{\faExternalLink*} that was ten times shorter, directly using the inductive datatype \lean{WellFounded}. This is the only result whose proof indicated in this paper differs significantly from the one used in the formalisation.
\end{remark}

\begin{lemma}\label{GoodProducts.span_iff_products}\href{https://github.com/leanprover-community/mathlib4/blob/ba9f2e5baab51310883778e1ea3b48772581521c/Mathlib/Topology/Category/Profinite/Nobeling.lean#L450-L469}{\faExternalLink*}
    If $\ev_S \left(P\right)$ spans $C(S,\Z)$, then $\ev_S \left(E(S)\right)$ spans $C(S,\Z)$. 
\end{lemma}
\begin{proof}
    It suffices to show that $\ev_S \left(P\right)$ is contained in the span of $\ev_S \left(E(S)\right)$. Suppose it is not, and let $p$ be the smallest element of $P$ whose evaluation is not in the span of $\ev_S \left(E(S)\right)$ (this $p$ exists by Lemma \ref{P_well_ordered}). Write $\ev_S (p)$ as a linear combination of evaluations of strictly smaller elements of $P$. By minimality of $p$, each term of the linear combination is in the span of $\ev_S\left(E(S)\right)$, implying that $p$ is as well, a contradiction. 
\end{proof}

\begin{lemma}\label{limit}\href{https://github.com/leanprover-community/mathlib4/blob/ba9f2e5baab51310883778e1ea3b48772581521c/Mathlib/Topology/Category/Profinite/Nobeling.lean#L229-L246}{\faExternalLink*}
    Let $F$ denote the contravariant functor from the (filtered) poset of finite subsets of $I$ to the category of profinite spaces, which sends $J$ to $S_J$. Then $S$ is homeomorphic to the limit of $F$. 
\end{lemma}
\begin{proof}[Proof sketch]
    Since $S$ is compact and the limit is Hausdorff, it suffices to show that the natural map from $S$ to the limit of $F$ induced by the projection maps $\pi_J : S \to S_J$ is bijective. 
    
    For injectivity, let $a, b \in S$ such that $\pi_J(a) = \pi_J(b)$ for all finite subsets $J$ of $I$. For all $i \in I$ we have $a_i = \pi_{\{i\}} (a) = \pi_{\{i\}} (b) = b_i$, hence $a = b$. 

    For surjectivity, let $b \in \operatorname{lim} F$. Denote by 
    \[
        f_J : \operatorname{lim} F \to S_J
    \] 
    the projection maps. We need to construct an element $a$ of $C$ such that $\pi_J(a) = f_J(b)$ for all $J$. In other words, we need to show that the intersection
    \[
        \bigcap_J \pi_J^{-1} \{f_J(b)\},    
    \]
    where $J$ runs over all finite subsets of $I$, is nonempty. By Cantor's intersection theorem \ref{cantor_inter}, it suffices to show that this family is directed (all the fibres are closed by continuity of the $\pi_J$, and closed subsets of a compact Hausdorff space are compact). To show that it is directed, it suffices to show that for $J \subseteq K$, we have 
    \[
        \pi_K^{-1}\{f_K(b)\} \subseteq \pi_J^{-1}\{f_J(b)\}
    \] 
    (by Lemma \ref{directed_of_sup}). This follows easily because the transition maps in the limits are just restrictions of the $\pi_J$. 
\end{proof}

\begin{lemma}\label{GoodProducts.spanFin}\href{https://github.com/leanprover-community/mathlib4/blob/ba9f2e5baab51310883778e1ea3b48772581521c/Mathlib/Topology/Category/Profinite/Nobeling.lean#L623-L675}{\faExternalLink*}
    Let $J$ be a finite subset of $I$. Then $\ev_{S_J}\left(E(S_J)\right)$ spans $C(S_J,\Z)$. 
\end{lemma}
\begin{proof}
    By lemma \ref{GoodProducts.span_iff_products}, it suffices to show that $\ev_{S_J}(P)$ spans. For $x \in S_J$, denote by $f_x$ the map $S_J \to \Z$ given by the Kronecker delta $f_x(y) = \delta_{xy}$.   
    
    Since $S_J$ is finite, the set of continuous maps is actually the set of all maps, and the maps $f_x$ span $C(S_J,\Z)$.  

    Now let $j_1 > \cdots > j_r$ be a decreasing enumeration of the elements of $J$. Let $x \in S_J$ and let $e_i$ denote $e_{S_J, j_i}$ if $x_{j_i} = 1$ and $(1 - e_{S_J, j_i})$ if $x_{j_i} = 0$. Then 
    \[
        f_{j_i} = \prod_{i = 1}^{r} e_i   
    \]
    is in the span of $P(S_J)$, as desired. 
\end{proof}

\begin{lemma}\href{https://github.com/leanprover-community/mathlib4/blob/ba9f2e5baab51310883778e1ea3b48772581521c/Mathlib/Topology/Category/Profinite/Nobeling.lean#L687-L696}{\faExternalLink*}
    $P$ spans $C(S,\Z)$. 
\end{lemma}
\begin{proof}
    Let $f \in C(S,\Z)$. Then by Lemmas \ref{limit} and \ref{jointly_surjective}, there is a $g \in C(S_J, \Z)$ such that $f = \pi_J^* (g)$. Writing this $g$ as a linear combination of elements of $E(S_J)$, by Lemma \ref{Products.evalFacProp} we see that $f$ is a linear combination of elements of $P$ as desired. 
\end{proof}

\subsection{Linear independence}\label{subsec_indep}

\begin{notation}
    Regard $I$ with its well-ordering as an ordinal. Let $Q$ denote the following predicate on an ordinal $\mu \leq I$: 
            
    For all closed subsets $S$ of $ \prod_{i \in I} \{0,1\}$, such that for all $x \in S$ and $i \in I$, $x_i = 1$ implies $i < \mu$, $E(S)$ is linearly independent in $C(S,\Z)$. 
\end{notation}

We want to prove the statement $Q(I)$. We prove by induction on ordinals that $Q(\mu)$ holds for all ordinals $\mu \leq I$.

\begin{lemma}\href{https://github.com/leanprover-community/mathlib4/blob/ba9f2e5baab51310883778e1ea3b48772581521c/Mathlib/Topology/Category/Profinite/Nobeling.lean#L1732-L1740}{\faExternalLink*}
    The base case of the induction, $Q(0)$, holds.
\end{lemma}
\begin{proof}
    In this case, $S$ is empty or a singleton. If $S$ is empty, the result is trivial. Suppose $S$ is a singleton. We want to show that $E(S)$ consists of only the empty list, which evaluates to $1$ and is linearly independent in $C(S,\Z) \cong \Z$. Let $p \in P$ and suppose $p$ is nonempty. Then it is strictly larger than the empty list. But the evaluation of the empty list is $1$, which spans $C(S,\Z) \cong \Z$, and thus $\ev_S(p)$ is in the span of strictly smaller products, i.e. not in $E(S)$.  
\end{proof}

\subsubsection{Limit case}

Let $\mu$ be a limit ordinal, $S$ a closed subset such that for all $x \in S$ and $i \in I$, $x_i = 1$ implies $i < \mu$. In other words, $S = S_\mu$. Suppose $Q(\mu')$ holds for all $\mu' < \mu$. Then in particular $E(S_{\mu'})$ is linearly independent 

\begin{lemma}\label{eval_image}\href{https://github.com/leanprover-community/mathlib4/blob/ba9f2e5baab51310883778e1ea3b48772581521c/Mathlib/Topology/Category/Profinite/Nobeling.lean#L951-L958}{\faExternalLink*}
    Let $\mu' < \mu$ and $p \in P$ whose entries are all $< \mu'$. Then 
    \[
        \pi_{\mu'}^* \left( \ev_{S_{\mu'}} \left( \{q \in P \suchthat q < p\} \right) \right) = 
        \ev_{S_{\mu}} \left( \{q \in P \suchthat q < p\} \right).
    \]
\end{lemma}
\begin{proof}
    If $q < p$, then every element of $q$ is also $< \mu'$. Thus, by lemma \ref{Products.evalFacProp}, 
    \[
        \pi_{\mu'}^*\left(\ev_{S_{\mu'}}(q)\right) = \ev_{S_\mu} (q),
    \]
    as desired.
\end{proof}

\begin{lemma}\label{Products.limitOrdinal}\href{https://github.com/leanprover-community/mathlib4/blob/ba9f2e5baab51310883778e1ea3b48772581521c/Mathlib/Topology/Category/Profinite/Nobeling.lean#L1074-L1088}{\faExternalLink*}
    \[E(S_\mu) = \bigcup_{\mu' < \mu} E(S_{\mu'}) \]
\end{lemma} 
\begin{proof}
    The inclusion from right to left follows from Lemma \ref{good_mono}, so we only need to show that if $p \in E(S_\mu)$ then there exists $\mu'< \mu$ such that $p \in E(S_{\mu'})$. Take $\mu'$ to be the supremum of the set $\{i+1 \suchthat i \in p\}$. Then $\mu' = 0$ if $p$ is empty, and of the form $i+1$ for an ordinal $i < \mu$ if $p$ is nonempty. In either case, $\mu' < \mu$. 
    
    Since every $i \in p$ satisfies $i < \mu' < \mu$, we have 
    \[
        \ev_{S_\mu}(p) = \pi_{\mu'}^*(\ev_{S_{\mu'}}(p))
    \] 
    and 
    \[
        \ev_{S_{\mu}} \left( \{q \in P \suchthat q < p\} \right) 
        = \pi_{\mu'}^* \left( \ev_{S_{\mu'}} \left( \{q \in P \suchthat q < p\} \right) \right)
    \]
    so if $\ev_{S_{\mu'}}(p)$ is in the span of 
    \[
        \ev_{S_{\mu'}} \left( \{q \in P \suchthat q < p\} \right),
    \] 
    then $\ev_{S_{\mu}}(p)$ is in the span of 
    \[
        \ev_{S_{\mu}} \left( \{q \in P \suchthat q < p\} \right), 
    \]
    contradicting the fact that $p \in E(S_\mu)$. 
\end{proof}

\begin{lemma}\label{GoodProducts.union}\href{https://github.com/leanprover-community/mathlib4/blob/ba9f2e5baab51310883778e1ea3b48772581521c/Mathlib/Topology/Category/Profinite/Nobeling.lean#L1090-L1102}{\faExternalLink*}
    \[
        \ev_{S_\mu}\left( E(S_\mu)\right) = \bigcup_{\mu' < \mu} \pi_{\mu'}^*\left( \ev_{S_{\mu'}} \left( E(S_{\mu'}) \right) \right)
    \]
\end{lemma}
\begin{proof}
    This follows from a combination of Lemmas \ref{Products.evalFacProp} and \ref{Products.limitOrdinal}.
\end{proof}

The family of subsets in the union in Lemma \ref{GoodProducts.union} is directed with respect to the subset relation (this follows from Lemmas \ref{directed_of_sup}, \ref{good_mono}, and \ref{Products.evalFacProp}). The sets $\ev_{S_{\mu'}}(E(S_{\mu'}))$ are all linearly independent by the inductive hypothesis, and by injectivity of $\pi_\mu^*$, their images under that map are as well. Thus, by Lemma \ref{linearIndependent_iUnion_of_directed}, the union is linearly independent, and we are done. \href{https://github.com/leanprover-community/mathlib4/blob/ba9f2e5baab51310883778e1ea3b48772581521c/Mathlib/Topology/Category/Profinite/Nobeling.lean#L1742-L1749}{\faExternalLink*}

\subsubsection{Successor case}
Let $\mu$ be an ordinal, $S$ a closed subset such that for all $x \in S$ and $i \in I$, $x_i = 1$ implies $i < \mu + 1$. In other words, $S = S_{\mu+1}$. Suppose $Q(\mu)$ holds. Then in particular $\ev_{S_\mu} : E(S_{\mu}) \to C(S_\mu, \Z)$ is linearly independent. 

To prove the inductive step in the successor case, we construct a closed subset $S'$ of $ \prod_{i \in I} \{0,1\}$ such that for all $x \in S'$, $x_i = 1$ implies $i < \mu$, and a commutative diagram
\begin{equation}
\label{comm_diagram}
\begin{tikzcd}
	0 & {C(S_\mu,\Z)} & {C(S,\Z)} & {C(S',\Z)} \\
	& {E(S_\mu)} & {E(S)} & {E'(S)}
	\arrow[from=1-1, to=1-2]
	\arrow["{\pi_\mu^*}", from=1-2, to=1-3]
	\arrow["{g}", from=1-3, to=1-4]
	\arrow[hook, from=2-2, to=2-3]
	\arrow["{\ev_S}", from=2-3, to=1-3]
	\arrow["{\ev_{S_\mu}}", from=2-2, to=1-2]
	\arrow[hook', from=2-4, to=2-3]
\end{tikzcd}
\end{equation}
where the top row is exact and $E'(S)$ is the subset of $E(S)$ consisting if those $p$ with $\mu \in p$ (note that $p$ necessarily starts with $\mu$). For $p \in P$, we denote by $p^t \in P$ the sequence obtained by removing the first element of $p$ ($t$ stands for \emph{tail}). The linear map $g$ has the property that $g(\ev_S(p)) = \ev_{S'}(p^t)$ and $p^t \in E(S')$. Given such a construction, the successor step in the induction follows from lemma \ref{linearIndependent_leftExact}.

\begin{construction}\href{https://github.com/leanprover-community/mathlib4/blob/ba9f2e5baab51310883778e1ea3b48772581521c/Mathlib/Topology/Category/Profinite/Nobeling.lean#L1193-L1226}{\faExternalLink*}
    Let 
    \[
        S_0 = \{x \in S \suchthat x_\mu = 0\},
    \]
    \[
        S_1 = \{x \in S \suchthat x_\mu = 1\},
    \]
    and 
    \[
        S' = S_0 \cap \pi_\mu(S_1).
    \]
    Then $S'$ satisfies the inductive hypothesis. 
\end{construction}

\begin{construction}\href{https://github.com/leanprover-community/mathlib4/blob/ba9f2e5baab51310883778e1ea3b48772581521c/Mathlib/Topology/Category/Profinite/Nobeling.lean#L1261-L1287}{\faExternalLink*}
    Let $g_0 : S' \to S$ denote the inclusion map, and let $g_1 : S' \to S$ denote the map that swaps the $\mu$-th coordinate to $1$ (since $S' \subseteq \pi_\mu(S_1)$, this map lands in $S$). These maps are both continuous, and we obtain a linear map 
    \[
        g_1^* - g_0^* : C(S, \Z) \to C(S', \Z),
    \] 
    which we denote by $g$. 
\end{construction}

\begin{lemma}\label{succ_exact}\href{https://github.com/leanprover-community/mathlib4/blob/ba9f2e5baab51310883778e1ea3b48772581521c/Mathlib/Topology/Category/Profinite/Nobeling.lean#L1366-L1376}{\faExternalLink*}
    The top row in diagram (\ref{comm_diagram}) is exact. 
\end{lemma}
\begin{proof}
    We already know that $\pi_\mu^*$ is injective. Also, since 
    \(
        \pi_\mu \circ g_1 = \pi_\mu \circ g_0,
    \) 
    we have 
    \[
    g \circ \pi_\mu^* = 0. 
    \]
    Now suppose we have 
    \[
    f \in C(S,\Z) \text{ with } g(f) = 0. 
    \]
    We want to find an 
    \[
    f_\mu \in C(S_\mu,\Z) \text{ with } f_\mu \circ \pi_\mu = f.
    \]
    Denote by 
    \[
    \pi_\mu' : \pi_\mu(S_1) \to S_1
    \]
    the map that swaps the $\mu$-th coordinate to $1$. Since $g(f) = 0$, we have 
    \[
        f \circ g_1 = f \circ g_0
    \] 
    and hence the two continuous maps $f_{|S_0}$ and $f_{|S_1} \circ \pi_{\mu}'$ agree on the intersection 
    \[
        S' = S_0 \cap \pi_\mu (S_1)
    \] 
    Together, they define the desired continuous map $f_\mu$ on all of $S_0 \cup \pi_\mu(S_1) = S_\mu$.   
\end{proof}

\begin{lemma}\label{Products.max_eq_eval}\href{https://github.com/leanprover-community/mathlib4/blob/ba9f2e5baab51310883778e1ea3b48772581521c/Mathlib/Topology/Category/Profinite/Nobeling.lean#L1539-L1567}{\faExternalLink*}
    If $p \in P$ starts with $\mu$, then $g(\ev_S(p)) = \ev_{S'}(p^t)$. 
\end{lemma}
\begin{proof}
    This follows from considering all the cases given by Lemma \ref{Products.eval_eq}. We omit the proof here and refer to the Lean proof linked above.
\end{proof}

\begin{remark}
    If $p \in E(S)$ and $\mu \in p$, then $p$ satisfies the hypotheses of Lemma \ref{Products.max_eq_eval}.
\end{remark}

\begin{lemma}\href{https://github.com/leanprover-community/mathlib4/blob/ba9f2e5baab51310883778e1ea3b48772581521c/Mathlib/Topology/Category/Profinite/Nobeling.lean#L1608-L1678}{\faExternalLink*}
    If $p \in E(S)$ and $\mu \in p$, then $p^t \in E(S')$.
\end{lemma}
\begin{proof}
    Contraposing the statement, it suffices to show that if
    \[
        \ev_{S'}\left(p^t\right) \in \operatorname{Span}\left( \ev_{S'}\left(\{q \suchthat q < p^t \}\right)\right),
    \]
    then 
    \[
        \ev_S(p) \in \operatorname{Span}\left(\ev_S \left(\{(q) \suchthat q < p \}\right)\right).
    \]
    Given a 
    $q \in P$ such that 
    $i \in q$ implies 
    $i < \mu$, we denote by 
    $q^\mu \in P$ the sequence obtained by adding 
    $\mu$ at the front. Write 
    \[
        g(\ev_S(p)) = \ev_{S'}(p^t)
        = \sum_{q < p^t} n_q \ev_{S'}(q)
        = \sum_{q < p^t} n_q g(\ev_{S}(q^\mu)).
    \]
    Then by Lemma \ref{succ_exact}, there exists an $n \in C(S_\mu,\Z)$ such that 
    \[
        \ev_S(p) = \pi_\mu^*(n) + \sum_{q < p^t} n_q (\ev_{S}(q^\mu)).
    \]
    Now it suffices to show that each of the two terms in the sum above is in the span of $\{\ev_{S}(q) \suchthat q < p \}$. The latter term is because $q < p^t$ implies $q^\mu < p$. The former term is because we can write $n$ as a linear combination indexed by $E(S_\mu)$, and for $q \in E(S_\mu)$ we have $\pi_\mu^*\left(\ev_{S_\mu}(q)\right) = \ev_S(q)$ and $\mu \notin q$ so $q < p$.
    
\end{proof}

\begin{lemma}\label{GoodProducts.union_succ}\href{https://github.com/leanprover-community/mathlib4/blob/ba9f2e5baab51310883778e1ea3b48772581521c/Mathlib/Topology/Category/Profinite/Nobeling.lean#L1390-L1412}{\faExternalLink*}
    The set $E(S)$ is the disjoint union of $E(S_\mu)$ and 
    \[
        E'(S) = \{p \in E(S) \suchthat \mu \in p\}.
    \]
\end{lemma}
\begin{proof}
    We already know by Lemma \ref{good_mono} that $E(S_\mu) \subseteq E(S)$. Also, as noted in Remark \ref{rem:prop_of_isGood}, if $p \in E(S_\mu)$ then all elements of $p$ are $<\mu$ and hence $p \notin E'(S)$. Now it suffices to show that if $p \in E(S) \setminus E'(S)$, then $p \in E(S_\mu)$. 
    
    Since $p \in E(S) \setminus E'(S)$, every $i \in p$ satisfies $i < \mu$. We have 
    \[
        \ev_{S}(p) = \pi_{\mu}^*(\ev_{S_{\mu}}(p))
    \] 
    and 
    \[
        \ev_{S} \left( \{q \in P \suchthat q < p\} \right) 
        = \pi_{\mu}^* \left( \ev_{S_{\mu}} \left( \{q \in P \suchthat q < p\} \right) \right)
    \]
    so if $\ev_{S_{\mu}}(p)$ is in the span of 
    \[
        \ev_{S_{\mu}} \left( \{q \in P \suchthat q < p\} \right), 
    \]
    then  $\ev_{S}(p)$ is in the span of 
    \[
        \ev_{S} \left( \{q \in P \suchthat q < p\} \right), 
    \]
    contradicting the fact that $p \in E(S)$. 
\end{proof}

The above lemmas prove all the claims made at the beginning of this section, concluding the inductive proof. \href{https://github.com/leanprover-community/mathlib4/blob/ba9f2e5baab51310883778e1ea3b48772581521c/Mathlib/Topology/Category/Profinite/Nobeling.lean#L1751-L1762}{\faExternalLink*}

\section{The formalisation}\label{sec:formalisation}

First a note on terminology: in the mathematical exposition of the proof in section \ref{sec:theorem}, we have talked about continuous maps from $S$ to $\Z$. Since $\Z$ is discrete, these are the same as the locally constant maps. The statement we have formalised is Listing \ref{list:nobeling}.

\begin{lstlisting}[caption={Nöbeling's theorem}, label=list:nobeling]
instance LocallyConstant.freeOfProfinite (S : Profinite.{u}) :
    Module.Free ℤ (LocallyConstant S ℤ)
\end{lstlisting}

which says that the $\Z$-module of locally constant maps from $S$ to $\Z$ is free. When talking about locally constant maps, one does not have to specify a topology on the target, which is slightly more convenient when working in a proof assistant.

The actual proof is about closed subsets of the product $ \prod_{i \in I}\{0,1\}$, which is of course the same thing as the space of functions $I \to \{0,1\}$. We implement it as the type \lean{I → Bool}, where \lean{Bool} is the type with two elements called \lean{true} and \lean{false}. This is the canonical choice for a two-element discrete topological space in Mathlib. 

\subsection{The implementation of $P$ and $E(S)$}\label{subsec:P_and_E}

We implemented the set $P$ as the type \lean{Products I} defined as
\begin{lstlisting}
def Products (I : Type*) [LinearOrder I] := {l : List I // l.Chain' (·>·)}
\end{lstlisting}

The predicate \lean{l.Chain' (·>·)} means that adjacent elements of the list \lean{l} are related by ``$>$''. We define the evaluation $\ev_S$ of products as 
\begin{lstlisting}
def Products.eval (S : Set (I → Bool)) (l : Products I) : 
    LocallyConstant S ℤ := (l.val.map (e S)).prod
\end{lstlisting}
where \lean{l.val.map (e S)} is the list of $e_{S,i}$ for $i$ in the list \lean{l.val}, and \lean{List.prod} is the product of the elements of a list. 

We define a predicate on \lean{Products}
\begin{lstlisting}
def Products.isGood (S : Set (I → Bool)) (l : Products I) : Prop :=
  l.eval S ∉ Submodule.span ℤ ((Products.eval S) '' {m | m < l})
\end{lstlisting}
and then the set $E(S)$ becomes 
\begin{lstlisting}
def GoodProducts (S : Set (I → Bool)) : Set (Products I) :=
  {l : Products I | l.isGood S}
\end{lstlisting}

It is slightly painful to prove completely trivial lemmas like \ref{Products.eval_eq} and its corollary \ref{Products.evalFacProp} in Lean. Indeed, these results are not mentioned in the proof of \cite[Theorem 5.4]{condensed}. Although trivial, they are used often in the proof of the theorem and hence very important to making the proof work. Reading an informal proof of this theorem, one might never realise that these trivialities are used. This is an example of a useful by-product of formalisation; more clarity of exposition.

\subsection{Ordinal induction}\label{subsec:ordinal}

When formalising an inductive proof of any kind, one has to be very precise about what statement one wants to prove by induction. This is almost never the case in traditional mathematics texts. For example, the proof of \cite[Theorem 5.4]{condensed} claims to be proving by induction that $E(S)$ is a basis of $C(S,\Z)$, not just that it is linearly independent. Furthermore, the set $I$ is not fixed throughout the inductive proof which makes it somewhat unclear what the inductive hypothesis actually says. Working inside the topological space $ \prod_{i \in I} \{0,1\}$ for a fixed set $I$ throughout the proof was convenient in the successor step. This avoided problems that are solved by abuse of notation in informal texts, such as regarding a set as the same thing as its image under a continuous embedding.

The statement of the induction principle for ordinals in Mathlib is the following\footnote{We have altered the notation slightly to match the notation in this paper.}:

\begin{lstlisting}
def Ordinal.limitRecOn {Q : Ordinal → Sort _} (o : Ordinal) 
    (H₁ : Q 0) 
    (H₂ : ∀ o, Q o → Q (succ o)) 
    (H₃ : ∀ o, IsLimit o → (∀ o' < o, Q o')→ Q o) : 
    Q o
\end{lstlisting}

In our setting, given a map \lean{Q : Ordinal → Prop} (in other words, a \emph{predicate on ordinals})\footnote{\lean{Prop} is \lean{Sort 0}}, we can prove $Q(\mu)$ for any ordinal $\mu$ if three things hold:
\begin{itemize}
    \item The zero case: $Q(0)$ holds.
    \item The successor case: for all ordinals $\lambda$, $Q(\lambda)$ implies $Q(\lambda + 1)$.
    \item The limit case: for every limit ordinal $\lambda$, if $Q(\lambda')$ holds for every $\lambda' < \lambda$, then $Q(\lambda)$ holds.
\end{itemize}
Finding the correct predicate $Q$ on ordinals was essential to the success of this project: 

\begin{lstlisting}
def Q (I : Type*) [LinearOrder I] [IsWellOrder I (·<·)] (o : Ordinal) : Prop :=
  o ≤ Ordinal.type (·<· : I → I → Prop) → 
    (∀ (S : Set (I → Bool)), IsClosed S → contained S o → 
      LinearIndependent ℤ (GoodProducts.eval S))
\end{lstlisting}

The inequality 
\begin{lstlisting}
o ≤ Ordinal.type (·<· : I → I → Prop)
\end{lstlisting}
means that $o \leq I$ when $I$ is considered as an ordinal, and the proposition \lean{contained S o} is defined as 
\begin{lstlisting}
def contained {I : Type*} [LinearOrder I] [IsWellOrder I (·<·)] 
    (S : Set (I → Bool)) (o : Ordinal) : Prop := 
  ∀ f, f ∈ S → ∀ (i : I), f i = true → ord I i < o
\end{lstlisting}
and \lean{ord I i} is an abbreviation for 
\begin{lstlisting}
Ordinal.typein (·<· : I → I → Prop) i
\end{lstlisting}
i.e. the element $i \in I$ considered as an ordinal. The conclusion 
\begin{lstlisting}
LinearIndependent ℤ (GoodProducts.eval S)
\end{lstlisting}
means that the map $\ev_S : E(S) \to C(S,\Z)$ is linearly independent.

As is often the case, this is quite an involved statement that we are proving by induction, and when writing informally, mathematicians wouldn't bother to specify the map \lean{Q : Ordinal → Prop} explicitly.

\subsection{Piecewise defined locally constant maps}\label{subsec:piecewise}

In the proof of Lemma \ref{succ_exact}, we defined a locally constant map $S_\mu \to \Z$ by giving locally constant maps from $S_0$ and $\pi_\mu(S_1)$ that agreed on the intersection, and noting that this gives a locally constant map from the union which is equal to $S_\mu$. To do this in Lean, the following definition was added to Mathlib \href{https://github.com/leanprover-community/mathlib4/blob/ba9f2e5baab51310883778e1ea3b48772581521c/Mathlib/Topology/LocallyConstant/Basic.lean#L613-L634}{\faExternalLink*}:

\begin{lstlisting}
def LocallyConstant.piecewise {X Z : Type*} [TopologicalSpace X] {C₁ C₂ : Set X}
    (h₁ : IsClosed C₁) (h₂ : IsClosed C₂) (h : C₁ ∪ C₂ = Set.univ)
    (f : LocallyConstant C₁ Z) (g : LocallyConstant C₂ Z) 
    (hfg : ∀ (x : X) (hx : x ∈ C₁ ∩ C₂), f ⟨x, hx.1⟩ = g ⟨x, hx.2⟩)
    [∀ j, Decidable (j ∈ C₁)] : LocallyConstant X Z where
  toFun i := if hi : i ∈ C₁ then f ⟨i, hi⟩
    else g ⟨i, (compl_subset_iff_union.mpr h) hi⟩
  isLocallyConstant := omitted
\end{lstlisting}
It says that given locally constant maps $f$ and $g$ defined respectively on closed subsets $C_1$ and $C_2$ which together cover the space $X$, such that $f$ and $g$ agree on $C_1 \cap C_2$, we get a locally constant map defined on all of $X$. This seems like exactly what we need in the above-mentioned proof. However, there is a subtlety, in that because of how the rest of the inductive proof is structured, we want the sets $S_0$, $\pi_\mu(S_1)$ and $S_\mu$ all to be considered as subsets of the underlying topological space $ \prod_{i \in I} \{0,1\}$. To use \lean{LocallyConstant.piecewise}, we would have to consider $S_\mu$ as the underlying topological space and $S_0$ and $\pi_\mu(S_1)$ as subsets of it. This is possible and is what was done initially, but a cleaner solution is to define a variant of \lean{LocallyConstant.piecewise}:
\begin{lstlisting}
def LocallyConstant.piecewise' {X Z : Type*} [TopologicalSpace X] 
    {C₀ C₁ C₂ : Set X}
    (h₀ : C₀ ⊆ C₁ ∪ C₂) (h₁ : IsClosed C₁) (h₂ : IsClosed C₂) 
    (f₁ : LocallyConstant C₁ Z) (f₂ : LocallyConstant C₂ Z) 
    [DecidablePred (· ∈ C₁)]
    (hf : ∀ x (hx : x ∈ C₁ ∩ C₂), f₁ ⟨x, hx.1⟩  = f₂ ⟨x, hx.2⟩) : 
    LocallyConstant C₀ Z
\end{lstlisting}
which satisfies the equations 
\begin{lstlisting}
lemma LocallyConstant.piecewise'_apply_left {X Z : Type*} [TopologicalSpace X]
    {C₀ C₁ C₂ : Set X} (h₀ : C₀ ⊆ C₁ ∪ C₂) 
    (h₁ : IsClosed C₁) (h₂ : IsClosed C₂) 
    (f₁ : LocallyConstant C₁ Z) (f₂ : LocallyConstant C₂ Z)
    [DecidablePred (· ∈ C₁)] 
    (hf : ∀ x (hx : x ∈ C₁ ∩ C₂), f₁ ⟨x, hx.1⟩ = f₂ ⟨x, hx.2⟩)
    (x : C₀) (hx : x.val ∈ C₁) :
    piecewise' h₀ h₁ h₂ f₁ f₂ hf x = f₁ ⟨x.val, hx⟩
\end{lstlisting}
and
\begin{lstlisting}
lemma LocallyConstant.piecewise'_apply_right {X Z : Type*} [TopologicalSpace X] 
    {C₀ C₁ C₂ : Set X} (h₀ : C₀ ⊆ C₁ ∪ C₂) 
    (h₁ : IsClosed C₁) (h₂ : IsClosed C₂) 
    (f₁ : LocallyConstant C₁ Z) (f₂ : LocallyConstant C₂ Z)
    [DecidablePred (· ∈ C₁)] 
    (hf : ∀ x (hx : x ∈ C₁ ∩ C₂), f₁ ⟨x, hx.1⟩ = f₂ ⟨x, hx.2⟩)
    (x : C₀) (hx : x.val ∈ C₂) :
    piecewise' h₀ h₁ h₂ f₁ f₂ hf x = f₂ ⟨x.val, hx⟩
\end{lstlisting}
Here $C_0, C_1,$ and $C_2$ are subsets of the same underlying topological space $X$; $C_1$ and $C_2$ are closed sets covering $C_0$, and $f_1$ and $f_2$ are locally constant maps defined on $C_1$ and $C_2$ respectively, such that $f_1$ and $f_2$ agree on the intersection. This fits the application perfectly and shortened the proof of Lemma \ref{succ_exact} considerably. Subtleties like this come up frequently, and can stall the formalisation process, especially when formalising general topology. When formalising Gleason's theorem \href{https://github.com/leanprover-community/mathlib4/pull/5634}{\faExternalLink*} (another result in general topology relevant to condensed mathematics, see \cite[Definition 2.4]{condensed} and \cite{gleason}), similar subtleties arose about changing the ``underlying topological space'' to a subset of the previous underlying topological space. 

The phenomenon that it is sometimes more convenient to formalise the definition of an object rather as a subobject of some bigger object is, of course, well known. It was noted in the context of group theory by Gonthier et al.~during the formalisation of the odd order theorem, see \cite[Section 3.3]{gonthier}.

\subsection{Reflections on the proof}

The informal proof in \cite{condensed} is about half a page; 21 lines of text. Depending on how one counts (i.e. what parts of the code count as part of the proof and not just prerequisites), the formalised proof is somewhere between 1500 and 3000 lines of Lean code. A big part of the difference is because of omissions in the proof in \cite{condensed}. 

A more fair comparison would be with the entirety of section \ref{sec:theorem} in this paper, which is an account of all the mathematical contents of the formalised proof. Still, there is quite a big difference, which is mostly explained by the pedantry of proof assistants, as discussed in subsections \ref{subsec:P_and_E}, \ref{subsec:ordinal}, and \ref{subsec:piecewise}.  

\subsection{Mathlib integration}\label{sec:mathlib}

As discussed in recent papers by Nash \cite{oliver} and Best et al. \cite{riccardo}, when formalising mathematics in Lean, it is desirable to develop as much as possible directly against Mathlib. Otherwise, the code risks going stale and unusable, while if integrated into Mathlib it becomes part of a library that is continuously maintained. 

The development of this project took place on a branch of Mathlib, all code being written in new files. This was a good workflow to get the formalisation done as quickly as possible, because if new code is put in the ``correct places'' immediately, one has to rebuild part of Mathlib to be able to use that code in other places, which can be a slow process if changes are made deep in the import hierarchy. 

The proof of Nöbeling's theorem described in this paper has now been fully integrated into Mathlib. An unusually large portion of the code was of no independent interest, which resulted in a pull request adding one huge file, which Johan Commelin and Kevin Buzzard kindly reviewed in great detail, improving both the style and performance of the code.

\section{Towards condensed mathematics in Mathlib}\label{sec:condensed_mathlib}
The history of condensed mathematics in Lean started with the \emph{Liquid Tensor Experiment} (LTE) \cite{LTE, LTE_challenge}. This is an example of a formalisation project that was in some sense too big to be integrated into Mathlib. Nevertheless, it was a big success in that it demonstrated the capabilities of Lean and its community by fully formalising the complicated proof of a highly nontrivial theorem about so-called liquid modules. Moreover, it provided a setting in which to experiment with condensed mathematics and find the best way to do homological algebra in Lean. As mentioned above, the goal of LTE was to formalise one specialised theorem. This is somewhat orthogonal to the goal of Mathlib which is to build a coherent, unified library of formalised mathematics. It is thus understandable that the contributors of LTE chose to focus on completing the task at hand instead of spending time on moving some parts of the code to Mathlib. Now that both LTE and the port of Mathlib to Lean 4 have been completed, we are seeing some important parts of LTE being integrated into Mathlib. 

The definition \href{https://github.com/leanprover-community/mathlib4/blob/ba9f2e5baab51310883778e1ea3b48772581521c/Mathlib/Condensed/Basic.lean#L46-L47}{\faExternalLink*} of a condensed object was recently added to Mathlib. During a masterclass on formalisation of condensed mathematics organised in Copenhagen in June 2023, participants collaborated, under the guidance of Kevin Buzzard and Adam Topaz, on formalising as much condensed mathematics as possible in one week (all development took place in Lean 4 and the goal was to write material for Mathlib). The code can be found in the masterclass GitHub repository \href{https://github.com/adamtopaz/CopenhagenMasterclass2023}{\faExternalLink*} and much of it has already made it into Mathlib. 

Profinite spaces form a rich category of topological spaces and there is more work other than Nöbeling's theorem to be done in Mathlib. Being the building blocks of condensed sets, it is important to develop a good API for profinite spaces in Mathlib. There, profinite spaces are defined as totally disconnected compact Hausdorff spaces. It is proved \href{https://github.com/leanprover-community/mathlib4/blob/ba9f2e5baab51310883778e1ea3b48772581521c/Mathlib/Topology/Category/Profinite/AsLimit.lean#L104-L105}{\faExternalLink*} that every profinite space can be expressed as a cofiltered limit (more precisely, over the poset of its discrete quotients). It is also proved \href{https://github.com/leanprover-community/mathlib4/blob/ba9f2e5baab51310883778e1ea3b48772581521c/Mathlib/Topology/Category/Profinite/Basic.lean#L306-L307}{\faExternalLink*} that the category of profinite spaces has all limits and that the forgetful functor to topological spaces preserves them \href{https://github.com/leanprover-community/mathlib4/blob/ba9f2e5baab51310883778e1ea3b48772581521c/Mathlib/Topology/Category/Profinite/Basic.lean#L302-L303}{\faExternalLink*}. From this we can extract the following useful theorem:

\begin{theorem}\label{profinite-profinite}
    A topological space is profinite if and only if it can be written as a cofiltered limit of finite discrete spaces. 
\end{theorem}

The story about profinite spaces as limits does not end there, though. Sometimes it is not enough to know just that some profinite space \emph{can} be written as \emph{a} limit, but rather that there is a specific limit formula for it. Lemma \ref{limit} gives one specific way of writing a compact subset of a product as a cofiltered limit, which can be useful. Another example can be extracted from \cite{discrete}. This is the fact that the identity functor on the category of profinite spaces is right Kan extended from the inclusion functor from finite sets to profinite spaces along itself. This gives another limit formula for profinite spaces, coming from the limit formula for right Kan extensions, and is useful when formalising the definition of solid abelian groups \cite{solid}. 

It can also be useful to regard the category of profinite spaces as the pro-category of the category of finite sets. The definition of pro-categories and this equivalence of categories would make for a nice formalisation project and be a welcome contribution to Mathlib.

\section{Conclusion and future work}\label{sec:conclusion}

By formalising Nöbeling's theorem, we have illustrated that the induction principle for ordinals in Mathlib can be used to prove nontrivial theorems outside the theory ordinals themselves. Another contribution is the detailed proof given in section \ref{sec:theorem}, and of course as mentioned before, it is an important step for the formalisation of condensed mathematics to continue. 

A natural next step in the formalisation of the theory of solid abelian groups is to port the code in \cite{discrete, solid} to Lean 4 and get it into Mathlib. Then one can put together the discreteness characterisation and Nöbeling's theorem to prove the structural results about $\Z[S]^\solid$, which would lead us one step closer to an example of a nontrivial solid abelian group in Mathlib.

More broadly, it is important to continue moving as much as possible of the existing Lean code about condensed mathematics (from the LTE and the Copenhagen masterclass) into Mathlib. 

\bibliography{references}

\end{document}